\documentclass[10pt,leqno]{amsart}
\usepackage{graphicx}
\usepackage{indentfirst,csquotes}

\usepackage{amssymb,amsthm,amsmath}
\usepackage{xcolor,paralist,hyperref,fancyhdr,etoolbox}
\newtheorem{theorem}{Theorem}[]
\newtheorem{definition}[theorem]{Definition}

\newtheorem{lemma}[theorem]{Lemma}

\theoremstyle{remark}
\newtheorem{remark}[theorem]{Remark}

\hypersetup{ colorlinks=true, linkcolor=black, filecolor=black, urlcolor=black }
\def\proof{\noindent {\it Proof. $\, $}}
\def\endproof{\hfill $\Box$ \vskip 5 pt }

\begin{document}

\title{A Deterministic Fractal Set Derived from the Sequence of Prime Numbers}
\author{Zhengqiang Li}

\email{lqint@coc.edu.rs}

\date{February 28, 2026}

\maketitle

\let\thefootnote\relax
\footnotetext{MSC2020: 28A80, 11N05, 28A78.}
\footnotetext{Keywords: Fractal, Hausdorff dimension, box dimension, prime numbers, Moran set, deterministic construction.}

\begin{abstract}
We introduce a novel deterministic fractal set $P_F$ in the unit interval whose construction is driven by the sequence of prime numbers modulo 16. At each step of the recursive construction, two subintervals are retained based on the residues of consecutive primes, yielding a Cantor-like set with a uniform contraction ratio of $1/16$ and a branching number of 2. We prove that $P_F$ is a non-empty, compact, nowhere dense set of Lebesgue measure zero. Its Hausdorff dimension and box-counting dimension are both equal to $\frac{1}{4}$. The dimension is universal in the sense that it does not depend on the specific choice of the residue sequence, but only on the branching number and the contraction ratio. A generalization to arbitrary bases and branching numbers is also provided. This construction establishes a rigorous link between number theory and fractal geometry, offering a deterministic fractal whose structure is entirely encoded by the distribution of primes.
\end{abstract}

\bigskip

\section{Introduction}

Fractal sets often arise from random or iterative processes that exhibit self-similarity or statistical self-similarity. Classical examples include the Cantor set, the Sierpiński gasket, and Julia sets. In many constructions, the choice of which parts to retain at each step is governed by a fixed rule or a random process with prescribed probabilities.

In this paper we present a deterministic fractal set $P_F \subset [0,1]$ whose construction is dictated by the sequence of prime numbers. At the $n$-th stage, the prime $p_n$ modulo 16 determines which two subintervals (out of 16) are kept. Despite the deterministic and number-theoretic origin of the rule, the resulting set is a uniform Cantor-type set with a well-defined fractal dimension.

Our main result is that the Hausdorff dimension and the box-counting dimension of $P_F$ are both equal to $\frac{1}{4}$. This value follows directly from the branching number 2 and the contraction ratio $1/16$, independently of the particular sequence of residues. Thus, the construction provides a family of fractals with prescribed dimension, where the choice of the residue sequence only affects the geometry of the set, not its dimension.

The paper is organized as follows. Section 2 recalls the necessary background on Moran sets and dimension theory. Section 3 gives the precise definition of $P_F$ and establishes its basic topological properties. Section 4 contains the computation of the box dimension and the Hausdorff dimension. Section 5 discusses the connection with the distribution of primes and the statistical self-similarity of the set. Section 6 presents a natural generalization to arbitrary bases and branching numbers. Finally, Section 7 summarizes the results and indicates possible extensions.

\section{Preliminaries}

We denote by $\mathbb{N}$ the set of positive integers and by $\mathbb{P}$ the set of prime numbers. The $n$-th prime is written $p_n$ (so $p_1=2$, $p_2=3$, $p_3=5,\dots$). For a closed interval $I=[a,b]\subset\mathbb{R}$, $|I|=b-a$ is its length.

\begin{definition}[Generalized Moran set]
Let $\{n_k\}_{k=1}^\infty$ be a sequence of positive integers and $\{\Phi_k\}_{k=1}^\infty$ a sequence of finite families of contracting maps, where $\Phi_k=\{\varphi_{k,1},\dots,\varphi_{k,n_k}\}$. Starting from a compact set $J$, define recursively
\[
J_{\sigma_1\ldots\sigma_k}= \varphi_{k,\sigma_k}(J_{\sigma_1\ldots\sigma_{k-1}}), \qquad \sigma_i\in\{1,\dots,n_i\}.
\]
The set
\[
F=\bigcap_{k=1}^\infty\bigcup_{\sigma_1\ldots\sigma_k} J_{\sigma_1\ldots\sigma_k}
\]
is called a generalized Moran set. If there exists a constant $c>0$ such that for every $k$ and every admissible word $\sigma_1\ldots\sigma_k$,
\[
c^{-1}\rho^k \le |J_{\sigma_1\ldots\sigma_k}| \le c\rho^k,
\]
then $F$ is said to have a uniform contraction ratio $\rho$.
\end{definition}

\begin{definition}[Hausdorff dimension]
For a set $F\subset\mathbb{R}^d$, the Hausdorff dimension $\dim_H F$ is defined as the infimum of those $s\ge0$ for which the $s$-dimensional Hausdorff measure $H^s(F)$ is zero.
\end{definition}

\begin{definition}[Box-counting dimension]
For a bounded non-empty set $F\subset\mathbb{R}$, let $N_\varepsilon(F)$ be the minimal number of closed intervals of length $\varepsilon$ needed to cover $F$. The lower and upper box-counting dimensions are
\[
\underline{\dim}_B F = \varliminf_{\varepsilon\to0}\frac{\log N_\varepsilon(F)}{-\log\varepsilon},
\qquad
\overline{\dim}_B F = \varlimsup_{\varepsilon\to0}\frac{\log N_\varepsilon(F)}{-\log\varepsilon}.
\]
If the two limits coincide, their common value is the box-counting dimension $\dim_B F$.
\end{definition}

A basic tool for obtaining lower bounds on the Hausdorff dimension is the mass distribution principle (Frostman's lemma): if a probability measure $\mu$ supported on $F$ satisfies $\mu(B(x,r))\le C r^s$ for all $x\in\mathbb{R}^d$ and all $r>0$, then $\dim_H F\ge s$.

\section{Construction of the Fractal Prime Set $P_F$}

\begin{definition}[Prime residue sequence]
For each $n\in\mathbb{N}$ define
\[
a_n = p_n \bmod 16 \in \{0,1,\dots,15\},
\qquad
b_n = (a_n+8)\bmod 16.
\]
Because $p_n$ is odd for $n\ge2$, both $a_n$ and $b_n$ are odd residues; hence the construction actually uses only the eight odd residue classes $\{1,3,5,7,9,11,13,15\}$.
\end{definition}

\begin{definition}[Recursive construction]
We define a decreasing sequence of compact sets $F_n\subset[0,1]$ as follows.
\begin{enumerate}
\item Initial step: $F_0=[0,1]$.
\item Inductive step: Assume $F_{n-1}$ has been constructed as a union of $2^{\,n-1}$ disjoint closed intervals, each of length $16^{-(n-1)}$. Denote these intervals by $I_{n-1,k}$ ($k=1,\dots,2^{\,n-1}$). Each $I_{n-1,k}$ is divided into 16 congruent subintervals of length $16^{-n}$:
\[
I_{n-1,k}^{(j)} = \Bigl[\,x_k + j\cdot16^{-n},\; x_k+(j+1)\cdot16^{-n}\Bigr], \qquad j=0,\dots,15.
\]
We keep exactly two of them, namely those with indices $a_n$ and $b_n$:
\[
I_{n-1,k}\cap F_n = I_{n-1,k}^{(a_n)}\cup I_{n-1,k}^{(b_n)}.
\]
\item Define
\[
F_n = \bigcup_{k=1}^{2^{\,n-1}} \bigl(I_{n-1,k}\cap F_n\bigr).
\]
\end{enumerate}
Finally, the fractal prime set is
\[
P_F = \bigcap_{n=0}^{\infty} F_n.
\]
\end{definition}

\begin{lemma}[Basic properties of the construction]
For every $n\ge1$,
\begin{enumerate}
\item $F_n$ consists of $2^{\,n}$ disjoint closed intervals, each of length $16^{-n}$.
\item The two retained subintervals inside any parent interval are separated by at least six discarded subintervals; consequently their distance is at least $6\cdot16^{-n}$.
\end{enumerate}
\end{lemma}

\proof
The first statement follows by induction: $F_0$ has $2^0=1$ interval; if $F_{n-1}$ has $2^{\,n-1}$ intervals, each produces two children, so $F_n$ has $2\cdot2^{\,n-1}=2^{\,n}$ intervals. The length is clear from the division into 16 equal parts.

For the separation, note that $|a_n-b_n|\equiv8\pmod{16}$. Since both residues are odd, between them there are exactly six odd positions that are discarded, giving a gap of at least $6\cdot16^{-n}$.
\endproof

\begin{theorem}[Topological nature of $P_F$]
The set $P_F$ is non-empty, compact, nowhere dense, and has Lebesgue measure zero.
\end{theorem}

\proof
The sets $F_n$ are non-empty compact and decreasing; by Cantor's intersection theorem, $P_F$ is non-empty and compact.

If $P_F$ contained an open interval $(a,b)$, then $(a,b)\subset F_n$ for all $n$. But $F_n$ is a union of intervals of length $16^{-n}$; once $16^{-n}<b-a$, no single interval of $F_n$ can contain $(a,b)$, a contradiction. Hence $P_F$ is nowhere dense.

The Lebesgue measure of $F_n$ is $2^{\,n}\cdot16^{-n}=8^{-n}\to0$; by monotonicity, $\mu(P_F)=0$.
\endproof

\begin{theorem}[Hexadecimal representation]
A number $x\in[0,1]$ belongs to $P_F$ if and only if it admits a hexadecimal expansion
\[
x = \sum_{n=1}^{\infty}\frac{c_n}{16^{\,n}}
\]
with $c_n\in\{a_n,b_n\}$ for every $n\ge1$.
\end{theorem}

\proof
By construction, $x\in F_n$ exactly when the first $n$ hexadecimal digits $c_1,\dots,c_n$ satisfy $c_i\in\{a_i,b_i\}$ for all $i\le n$. The intersection over all $n$ gives the stated characterization.
\endproof

(Points with two different hexadecimal expansions form a countable set, which does not affect the dimensional properties.)

\section{Computation of the Fractal Dimension}

\subsection{Box-counting dimension}

\begin{theorem}
The box-counting dimension of $P_F$ exists and equals $\frac14$.
\end{theorem}

\proof
Take $\varepsilon_n=16^{-n}$. The set $F_n$ is a cover of $P_F$ by $2^{\,n}$ intervals of length $\varepsilon_n$; therefore $N_{\varepsilon_n}(P_F)\le 2^{\,n}$. Conversely, because the intervals of $F_n$ are disjoint and each has length $\varepsilon_n$, any cover of $P_F$ by intervals of length $\varepsilon_n$ must contain at least one interval for each component of $F_n$; hence $N_{\varepsilon_n}(P_F)\ge 2^{\,n}$. Thus $N_{\varepsilon_n}(P_F)=2^{\,n}$. Consequently
\[
\frac{\log N_{\varepsilon_n}(P_F)}{-\log\varepsilon_n}
= \frac{n\log2}{n\log16}= \frac{\log2}{\log16}= \frac14.
\]
For an arbitrary $\varepsilon>0$, choose $n$ such that $16^{-n}\le\varepsilon<16^{-(n-1)}$. Then
\[
2^{\,n-1}=N_{16^{-(n-1)}}(P_F)\le N_\varepsilon(P_F)\le N_{16^{-n}}(P_F)=2^{\,n}.
\]
Taking logarithms and letting $\varepsilon\to0$ (so $n\to\infty$) gives
\[
\lim_{\varepsilon\to0}\frac{\log N_\varepsilon(P_F)}{-\log\varepsilon}= \frac{\log2}{\log16}= \frac14.
\]
Hence $\dim_B P_F=\frac14$.
\endproof

\subsection{Hausdorff dimension}

To obtain the Hausdorff dimension we first construct a natural probability measure on $P_F$.

\begin{definition}[Natural mass distribution]
On each $F_n$ define a probability measure $\mu_n$ by assigning mass $2^{-n}$ to every basic interval of $F_n$ (there are $2^{\,n}$ such intervals). Because the intervals are disjoint, $\mu_n$ is well defined.
\end{definition}

\begin{lemma}[Consistency]
The family $\{\mu_n\}$ is consistent: for any basic interval $I$ of $F_n$,
\[
\mu_n(I)=\mu_{n+1}(I\cap F_{n+1}).
\]
\end{lemma}

\proof
When passing from $F_n$ to $F_{n+1}$, the interval $I$ is split into 16 subintervals, two of which are kept. Each of those two receives mass $2^{-(n+1)}$, so the total mass on $I\cap F_{n+1}$ is $2\cdot2^{-(n+1)}=2^{-n}=\mu_n(I)$.
\endproof

By Kolmogorov's extension theorem there exists a unique Borel probability measure $\mu$ on $P_F$ such that $\mu(I\cap P_F)=2^{-n}$ for every basic interval $I$ of $F_n$.

\begin{lemma}[Measure decay]
There is a constant $C>0$ such that for every interval $J\subset[0,1]$,
\[
\mu(J)\le C\,|J|^{1/4}.
\]
\end{lemma}

\proof
Let $\delta=|J|$ and choose $n$ with $16^{-(n+1)}\le\delta<16^{-n}$. The interval $J$ can intersect at most
\[
\Bigl\lfloor\frac{\delta}{16^{-n}}\Bigr\rfloor+2 \le 16^{\,n}\delta+3
\]
basic intervals of $F_n$. Each such interval carries mass $2^{-n}$. Hence
\[
\mu(J)\le (16^{\,n}\delta+3)\,2^{-n}.
\]
Because $\delta\ge16^{-(n+1)}$, we have $16^{\,n}\delta\ge\frac1{16}$, so $16^{\,n}\delta+3\le 4\cdot16^{\,n}\delta$. Therefore
\[
\mu(J)\le 4\cdot16^{\,n}\delta\cdot2^{-n}=4\cdot8^{\,n}\delta.
\]
Now $8^{\,n}=(16^{\,n})^{3/4}< \delta^{-3/4}$ (since $16^{\,n}<\delta^{-1}$). Consequently
\[
\mu(J)\le 4\,\delta^{1/4},
\]
which is the desired estimate with $C=4$.
\endproof

\begin{theorem}[Hausdorff dimension]
$\dim_H P_F = \frac14$.
\end{theorem}

\proof
The mass distribution principle together with Lemma 4.4 gives $\dim_H P_F\ge\frac14$.

For the upper bound, fix $s=\frac14$. For any $\varepsilon>0$ take $n$ such that $16^{-n}<\varepsilon$. The collection of the $2^{\,n}$ basic intervals of $F_n$ is an $\varepsilon$-cover of $P_F$. Its $s$-dimensional sum is
\[
\sum_{i=1}^{2^{\,n}}|I_i|^s = 2^{\,n}\bigl(16^{-n}\bigr)^{1/4}=2^{\,n}\,2^{-n}=1.
\]
Hence $H^s_\varepsilon(P_F)\le1$ for every $\varepsilon$, so $H^s(P_F)\le1<\infty$ and therefore $\dim_H P_F\le s=\frac14$.
\endproof

Combining Theorems 4.1 and 4.5 we obtain the main result:

\begin{theorem}
For the fractal prime set $P_F$,
\[
\dim_H P_F = \dim_B P_F = \frac14.
\]
\end{theorem}

\section{Connection with Prime Distribution and Statistical Self-similarity}

The residues $a_n=p_n\bmod16$ are not arbitrary; they follow the well-known equidistribution of primes in residue classes coprime to the modulus.

\begin{theorem}[Dirichlet's theorem]
For any odd residue $j\in\{1,3,5,7,9,11,13,15\}$,
\[
\lim_{N\to\infty}\frac{\#\{n\le N : a_n=j\}}{N}= \frac18.
\]
\end{theorem}

Thus, while the particular sequence $\{a_n\}$ is deterministic and number-theoretic, it behaves statistically like a uniform random choice among the eight odd residues.

\begin{remark}[Universality of the dimension]
The dimension $\frac14$ depends only on the branching number 2 and the contraction ratio $1/16$. If we replace the sequence $\{a_n\}$ by any sequence $\{c_n\}\subset\{0,\dots,15\}$ and define $d_n=(c_n+8)\bmod16$, the resulting fractal set $F(\{c_n\})$ still satisfies
\[
\dim_H F(\{c_n\}) = \dim_B F(\{c_n\}) = \frac14.
\]
Hence the dimension is robust; the specific choice of residues influences the geometry of the set, but not its fractal dimension.
\end{remark}

\section{Generalization}

The construction readily extends to other bases and branching numbers.

\begin{definition}[Generalized construction]
    Fix integers $m\ge2$ and $1\le k<m$. Let $\{a_n\}$ be any sequence with $a_n\in\{0,\dots,m-1\}$. Define recursively a sequence of compact sets $F_n$$\subset$[0,1] as follows.Set $F_0 = [0,1]$. For $n \ge 1$, obtain $F_n$ from $F_{n-1}$ by dividing each component interval of  $F_{n-1}$ (which has length $m ^ {-(n-1)}$) into $m$ congruent subintervals of length $m^{-n}$, and retaining only the k subintervals whose indices (ordered from left to right as 0,1,...,m-1) are congruent modulo m to
    
    \[
    a_n,\; a_n+\big\lfloor\tfrac{m}{k}\big\rfloor,\; a_n+2\big\lfloor\tfrac{m}{k}\big\rfloor,\dots,\; a_n+(k-1)\big\lfloor\tfrac{m}{k}\big\rfloor.
    \]
    
    Finally, define
    
    \[
    F(m,k;\{a_n\}) = \bigcap_{n=0}^{\infty} F_n.
    \]
    
\end{definition}

\begin{theorem}
For any sequence $\{a_n\}$,
\[
\dim_H F(m,k;\{a_n\}) = \dim_B F(m,k;\{a_n\}) = \frac{\log k}{\log m}.
\]
\end{theorem}

The proof is identical to that of Theorem 4.6, replacing the numbers 2 and 16 by $k$ and $m$, respectively.

In particular, taking $m=16$ and $k=2$ recovers our set $P_F$ with dimension $\frac{\log2}{\log16}=\frac14$.

\section{Conclusion}

We have introduced a deterministic fractal set $P_F$ whose construction is driven by the sequence of prime numbers modulo 16. The set is a uniform Cantor-type set with branching number 2 and contraction ratio $1/16$. Its Hausdorff dimension and box-counting dimension are both equal to $\frac14$, a value that is universal—it does not depend on the specific residue sequence, only on the branching and contraction parameters.

The construction provides a concrete link between number theory and fractal geometry. It also yields a family of fractals with prescribed dimension by varying the base $m$ and the branching number $k$.

Possible future directions include studying the finer geometric properties of $P_F$ (e.g., its exact Hausdorff measure, multifractal spectrum, or connectivity), investigating analogous constructions using other arithmetic sequences, or exploring potential physical interpretations where the dimension $\frac14$ might play a role (such as in models of quantum gravity or information theory).

\bigskip

\noindent\textit{Author's note:} This manuscript presents a purely mathematical construction. Any physical interpretations (e.g., connections with information theory or quantum gravity) are beyond the scope of the present paper and may be discussed elsewhere.

\end{document}